\documentclass{article}

\usepackage{amssymb}

\usepackage{amsfonts, amsmath, amsthm}

\newcommand{\oper}[1]{\operatorname{#1}}

\newtheorem{theorem}{Theorem}[section]
\newtheorem{lemma}[theorem]{Lemma}
\newtheorem{e-proposition}[theorem]{Proposition}

\newtheorem{e-definition}[theorem]{Definition\rm}

\begin{document}

\title{A Density Result for Real Hyperelliptic Curves}

\author{Brian Lawrence}

\date{April 2016}

\maketitle

\begin{abstract}

Let $\{\infty^+, \infty^-\}$ be the two points above $\infty$ on the real hyperelliptic curve $H: y^2 = (x^2 - 1) \prod_{i=1}^{2g} (x - a_i)$.  We show that the divisor $([\infty^+] - [\infty^-])$ is torsion in $\oper{Jac} J$ for a dense set of $(a_1, a_2, \ldots, a_{2g}) \in (-1, 1)^{2g}$.  In fact, we prove by degeneration to a nodal $\mathbb{P}^1$ that an associated period map has derivative generically of full rank. 
\end{abstract}

\section{Introduction}
\label{}

Our goal is to prove a density result, suggested by J-P.\ Serre, relating to the question of when a certain divisor on a real hyperelliptic curve is a torsion divisor.  The key difficulty is to show that a certain Jacobian derivative, relating to integrals on the curve, generically has full rank.  It would suffice to show that the Jacobian has full rank for even a single curve in the family.  We degenerate the curve to a $\mathbb{P}^1$ with nodes, and perform the calculation there.

The question about real hyperelliptic curves arises from the study of Chebyshev-like polynomials for unions of intervals on the real line.  For the connection between Chebyshev polynomials and integrals on hyperelliptic curves, see Chapter 2 of \cite{Bog}.  Our result guarantees the existence of hyperelliptic curves with certain periods rational, which can be used to define Chebyshev polynomials, as in Theorem 2.1 of \cite{Bog}.   As Serre points out, this result can also be used to complete a difficult passage in a paper by Robinson \cite{Rob}: the claim at the end of \S 4 that one can ``vary the intervals a little'' so that ``the heights of the corresponding slits [are] rational multiples of $\pi$'' is in fact equivalent to this density result for hyperelliptic curves.

\section{Prior Work}

Since writing this paper, I have learned that the main result (Theorem \ref{mainthm}) has appeared multiple times in the literature, with different proofs.

The result is Theorem 5 of \cite{Bog2}; additionally, Bogatyrev's paper gives a very explicit geometric description of the moduli space of real hyperelliptic curves and the solutions to Abel's equations.

The result is also proved as Theorem 2.1 of \cite{Tot}, with an application to bounding derivatives of polynomials in terms of their values on an interval.

Bogatyrev and Totik give independent proofs that the Jacobian of Lemma 4.1 of \cite{Law} is surjective at every point of the moduli space.
This is stronger than the result of \cite{Law}, where it is merely shown that the Jacobian is generically surjective.

Additionally, the result appears as the main result of \cite{Peh}, in the following form: any finite union $E$ of real disjoint intervals
can be approximated by a set of the form $E' = \mathcal{T}^{-1}([-1, 1])$, with $\mathcal{T}$ a polynomial.
The set $E'$ is obtained constructively by continuous deformation of a minimal polynomial.

I'd like to thank Andrey Bogatyrev for bringing these results to my attention.

\section{The Setup}

Fix a positive integer $g$ and let $\mathcal{U}$ denote the set of $2g$-tuples $a = (a_1, a_2, \ldots, a_{2g})$ of real numbers satisfying $-1 < a_1 < \cdots < a_{2g} < 1$.  For any $a \in \mathcal{U}$ the curve $P_a$ in $\mathbb{P}^2$ given by
$$y^2 = f(x) := (x^2-1) \prod_i (x - a_i)$$
has one point $\infty$ at infinity, at which it is singular; call its normalization $H_a$.  This is a real algebraic curve; it has two points lying over $\infty$, both defined over $\mathbb{R}$, and distinguished by the sign of $\frac{y}{x^{g+1}}$.  Call them $\infty^+$ and $\infty^-$.  
We say that $a$ (or $H_a$) is of \emph{torsion type} if the image of the divisor
$$[\infty^+] - [\infty^-]$$
in $\oper{Jac} H_a$ is torsion.  (This condition is independent of base field and can be checked after base change to $\mathbb{C}$.)

Our goal is to prove the following theorem.  (The problem, as well as the method of proof, was suggested in a lecture by Serre at Leiden in November 2015.)

\begin{theorem}
\label{mainthm}
The set of elements $a \in \mathcal{U}$ of torsion type is dense in $\mathcal{U}$.
\end{theorem}

Recall that we can detect whether $a$ is of torsion type by integrating differentials on $H_a$, as follows.

The global differentials on $H_a$ (i.e.\ the global sections of $\Omega_{H_a/\mathbb{R}}$) form a real vector space $\Omega$ of dimension $g$, whose complexification is the isomorphic to the cotangent space of the $\oper{Jac} (H_a)_\mathbb{C}$.  Explicitly, the global differentials are of the form
$$\frac{p(x) dx}{y},$$
with $p$ a polynomial of degree at most $g-1$.  Suppose
$$\omega_i = \frac{p_i(x) dx}{y},$$
for $1 \leq i \leq g$, form a basis $\mathcal{B}$ for this vector space.  (We will make an explicit choice of $\omega_i$, including the implicit choice of square root, later.)  Then we can compute the \emph{matrix of real periods} $M$ of $H_a$ with respect to $\mathcal{B}$, whose $(i, j)$-th entry is
$$2 \int_{a_{2j-1}}^{a_{2j}} \omega_i.$$
Here $i$ and $j$ each range from $1$ to $g$.  (We may write $M_\mathcal{B}(a)$ to emphasize the dependence of $M$ on $a$ and $\mathcal{B}$.)

The Jacobian of the base change of $H_a$ to $\mathbb{C}$ is a complex torus, which is naturally the quotient of the complexification $\Omega_\mathbb{C}$ by a rank-$2g$ integral lattice $\Lambda$.  This $\Lambda$ contains the rank-$g$ sublattice $\Lambda_\mathbb{R}$ generated (after choice of basis $\mathcal{B}$) by the columns of $M$; it follows that $\Lambda_\mathbb{R}$ is exactly the intersection of $\Lambda$ with the real subspace $\Omega$.

The coordinates of (a lift of) the divisor $[\infty^+] - [\infty^-]$ in this torus are given by the following $g$ integrals:
$$v_i = 2 \int_1^\infty \omega_i.$$
Let $v$ denote the vector of the $v_i$'s.  Note that $v$ in fact lies in the real subspace $\Omega$.  (Again, we may write $v_\mathcal{B}(a)$.)

Now $a$ is of torsion type if and only if some nonzero integral multiple of $v$ lies in the lattice $\Lambda$, or equivalently in $\Lambda_\mathbb{R}$.  Hence we have the following characterization of $a$ of torsion type.

\begin{lemma}
Suppose given some $a \in \mathcal{U}$, and make an arbitrary choice of basis $\mathcal{B}$ for $\Omega$.  Then $a$ is of torsion type if and only if the vector $M^{-1}_\mathcal{B}(a) v_\mathcal{B}(a)$ has all its components rational.
\end{lemma}

We remark in passing that, while $M$ and $v$ both depend on the choice of basis $\mathcal{B}$, the product $M^{-1} v$ appearing in the lemma above does not.  By a slight abuse of notation, we will use $M^{-1} v$ to denote this function
$$M^{-1} v: \mathcal{U} \rightarrow \mathbb{R}^g,$$
without reference to $\mathcal{B}$.

\section{Expanding $\mathcal{U}$}
Let $\mathcal{U}'$ denote the set of tuples $(a_1, a_2, \ldots, a_{2g+2})$ of real numbers such that
$$ -1 < \min (a_1, a_2) \leq \max (a_1, a_2) < \min (a_3, a_4) \leq \cdots \leq \max (a_{2g-1}, a_{2g}) < 1;$$
we equip $\mathcal{U}'$ with the real analytic structure induced from $\mathbb{R}^{2g}$.  Then $\mathcal{U}$ is naturally a subset of $\mathcal{U}'$.

We wish to extend $M$ and $v$ to functions on $\mathcal{U}'$.  As before, we will choose a basis $\mathcal{B}$ of the relevant space of differentials; write the basis elements as
$$\omega_i = \frac{p_i(x) dx}{\sqrt{f(x)}},$$
with each $p_i$ a polynomial of degree at most $g-1$.  (The choice of square root will be specified below.)  We will define $M$ and $v$ with respect to $\mathcal{B}$.

The definition of $M$ is problematic, since the endpoints of the integral
$$2 \int_{a_{2j-1}}^{a_{2j}} \omega_i$$
could coincide.  However this integral, where defined, is equal to a path integral in the complex plane
$$\int_{\gamma_j} \omega_i,$$
where $\gamma_j$ is a loop that goes around the two roots $a_{2j}$ and $a_{2j+1}$.  We now make this idea precise.

In order to choose a square root $\sqrt{f(x)}$, we make $g+2$ branch cuts in the complex plane.  First, cut along the intervals $(- \infty, -1]$ and $[1, \infty)$ of the real line.  This is exactly the locus where $(x^2 - 1)$ is positive real; away from this locus we may specify a square root $\sqrt{x^2-1}$ by requiring that its imaginary part be positive.  Similarly, each quadratic factor $(x-a_{2j-1})(x-a_{2j})$ has a holomorphic square root on the complement of the line segment joining $a_{2j-1}$ to $a_{2j}$; we choose the square root to be positive real for real $x$ outside the interval removed.  (If the two roots $a_{2j-1}$ and $a_{2j}$ coincide we merely puncture the plane at this point.)  The condition on the real parts of the roots implies that the branch cuts are disjoint.  Away from these branch cuts we have specified a choice of square root of $f$.

We define the paths $\gamma_j$ as follows, for $1 \leq j \leq g$.  Each $\gamma_j$ will be a loop that goes once counterclockwise around the branch cut from $a_{2j}$ to $a_{2j+1}$.  Precisely, one may take $\gamma_j$ to be a rectangle; its left-hand side is between $\max (a_{2j-3}, a_{2j-2})$ and $\min (a_{2j-1}, a_{2j})$, and its right-hand side between $\max (a_{2j-1}, a_{2j})$ and $\min (a_{2j+1}, a_{2j+2})$.  (When $j=1$ the left-hand side should be between $a_1 = -1$ and $\min (a_{1}, a_{2})$; and similarly for the right-hand side when $j=g$.)  The lower and upper edges of the rectangle lie in the lower and upper half-planes, respectively.  We also take the lower and upper edges to be symmetric about the real axis; this shows that the periods are real.

We define the matrix of periods to be the matrix with entries
$$M_{ij} = \int_{\gamma_j} \omega_i,$$
where $i$ and $j$ each range from $1$ to $g$.

First, we verify that this matrix $M$ is invertible.  When the roots $a_i$ are distinct, this is a consequence of the Hodge decomposition, applied to our hyperelliptic curve.  In the general case, our condition on the roots $a_i$ implies that $f$ has at most double roots; suppose there are $k$ double roots, and $2g+2-2k$ single roots.  In this case we again take $P_a$ to be the projectivized plane curve $y^2 = f(x)$, which has a singular point $\infty$ at its single intersection with the line at infinity in the projective plane.  We construct $H_a$ to be a ``desigularization of $P_a$ at $\infty$'': our $H_a$ is a cover of $P_a$ such that the map $H_a \rightarrow P_a$ is an isomorphism away from $\infty$.  We adjoin $y / x^{g+1}$ to the coordinate ring of $P_a$ in affine neighborhoods of $\infty$.  The resulting curve $H_a$ has again points $\infty^\pm$ lying over $\infty$, and is smooth at those two points; but $H_a$ retains the $k$ singular points of $P_a$ in the finite part of the projective plane.  Let $C$ be the normalization of $H_a$.  This $C$ is a curve of genus $g-k$, and $H_a$ is obtained from $C$ by glueing $k$ pairs of complex-conjugate points.  The vector space $\Omega$ of real global differentials on $H_a$ fits in an exact sequence
$$0 \rightarrow \Omega^1_C \rightarrow \Omega \rightarrow V_k \rightarrow 0,$$
where $V$ is a $k$-dimensional vector space over $\mathbb{R}$, and $\Omega^1_C$ denotes the matrix of real global differentials on $C$.  The map $\Omega \rightarrow V$ is the residue map.  Specifically, each singular point $x$ of $H_a$ has two preimages $x_1$ and $x_2$ in $C$; a global differential on $H_a$ pulls back to a differential on $C$ with at most simple poles at $x_1$ and $x_2$, and opposite residues at these two poles.  For each of the $k$ singular points $x_1, x_2, \ldots, x_k$, we choose one of its two preimages; the residues at these $k$ points give the map from $\Omega$ to $V$.

We may check invertibility of the period matrix with respect to any basis for $\Omega$.  Let us choose a basis containing $g-k$ elements $\omega_{k+1}, \omega_{k+2}, \ldots, \omega_{g}$ of $\Gamma(\Omega^1_C)$, and then an additional $k$ differentials $\omega_1, \omega_2, \ldots, \omega_k$ such that the pullback of $\omega_i$ to $C$ has nonzero residue around the preimages of $P_i$, but zero residue around the preimages of $P_j$ for $j \neq i$.  Furthermore, permute the indices of the loops $\gamma_j$ so that $\gamma_1, \gamma_2, \ldots, \gamma_k$ are the loops around the singular points.

Then for $1 \leq j \leq k$, the loop $\gamma_j$ lifts to a trivial loop on $C$ that goes around one of the preimages of $P_i$.  So, if $i > k$ then we have
$$M_{ij} = \int_{\gamma_j} \omega_i = 0.$$
In other words, the matrix $M$ is upper-triangular in this basis, so its determinant depends only on the two diagonal blocks.

The block $1 \leq i, j \leq k$ is diagonal with diagonal entries nonzero (given by the residue of $\omega_i$ around a preimage of $P_i$).  The other block has nonzero determinant by the Hodge decomposition applied to the curve $C$.  Hence, the matrix $M$ is invertible.

The definition of $v$ via
$$v_i = 2 \int_{1}^\infty \omega_i$$
remains valid on $\mathcal{U}'$ since $f$ cannot vanish at any $x>1$.

Next we show that $M$ and $v$ are real-analytic as functions on $\mathcal{U}'$, taking for $\mathcal{B}$ the standard basis
$$\left\{ \frac{x^i dx} {\sqrt{f(x)}} \right\}.$$

The integral defining $v$ can be replaced with a contour integral in the complex plane.  Keeping our previous branch cuts, integrate along a path $\gamma$ that starts from infinity in the lower half-plane, passes just to the left of $1$, say through $1-\epsilon$, and goes back to infinity through the upper half-plane.  To avoid integrating along a path that goes to infinity, we perform the change of variables $z = 1/x$.  (Assume $\gamma$ was chosen not to pass through $0$.)  Call $\gamma_0$ the corresponding path in the $z$-plane.  Now the two endpoints of $\gamma_0$ coincide at $0$, but in fact the integrand has different values at the two endpoints since the points are on different branches of our double cover of $\mathbb{P}^1$; to deal with this, split the path $\gamma_0$ into paths $\gamma_1$ (from $0$ to $1/(1-\epsilon)$, through the lower half-plane) and $\gamma_2$ (from $1/(1-\epsilon)$ to $0$, through the upper half-plane).  Since
$$\frac{x^i dx} {\sqrt{f(x)}}$$
is a holomorphic global differential on a double cover of $\mathbb{P}^1$, we may write it as
$$g_e(z) dz,$$
with $g_e$ holomorphic on a neighborhood of $\gamma_e$, for $e$ equal to $1$ or $2$.  Also, each $g_e$ is holomorphic in the coefficients of $f$.  The analyticity of $v$ (and of $M$) now follows from the following well-known lemma.

\begin{lemma}
\label{anal}
Let $\gamma$ be a piecewise-differentiable path in $\mathbb{C}$, and let $U$ be an open set containing $\gamma$.  Suppose we have an open set $D \subseteq \mathbb{C}^n$ and a holomorphic function $h$ on $D \times U$.  Then
$$H(z_1, z_2, \ldots, z_n) = \int_\gamma h(z_1, z_2, \ldots, z_n, w) dw$$
is holomorphic on $D$.
\end{lemma}

\begin{proof}
By Osgood's lemma, it is enough to check that $H$ is holomorphic in the variables $z_i$ taken one at a time, so we may suppose $n=1$ and write $z$ for $z_1$.  By Morera's theorem, it is enough to show that
$$\int_{\gamma '} H(z) dz = 0$$
over any loop $\gamma '$ whose interior is contained in $D$.  But this follows from the holomorphy of $h$, since Fubini's theorem allows one to interchange the order of integration.
\end{proof}

We apply the lemma to our situation as follows.  Fix real numbers $a_i^{(0)}$, and a path $\gamma$, as in the discussion preceding the lemma. We will apply the lemma to the integral
$$\int_{\gamma} \frac{x^j dx} {\sqrt{ (x^2 - 1) \prod_{i=1}^{2g} (x - a_i)}},$$
where we make a choice of square root as discussed above.
In the lemma we take $n = 2g$ with $z_i = a_i$, and
$$h(a_1, \ldots, a_{2g}, x) = \frac{x^j} {\sqrt{ (x^2 - 1) \prod_{i=1}^{2g} (x - a_i)}}.$$
We will take $D$ a neighborhood of $(a_1^{(0)}, \ldots, a_{2g}^{(0)})$ in $\mathbb{C}^n$, and $U$ a neighborhood of $\gamma$ in $\mathbb{C}$, such that $h$ remains holomorphic on $D \times U$ (to which our choice of square root extends).


The lemma shows that our integral is holomorphic in the variables $a_i$ in a neighborhood of an arbitrary $(a_1^{(0)}, \ldots, a_{2g}^{(0)}) \in \mathcal{U}'$; if we restrict $a_i$ to real values, the integral remains real-analytic, as desired.

We have now extended $M$ and $v$ to $\mathcal{U}'$.  As before we find that the product $M^{-1} v$ is independent of $\mathcal{B}$ and gives a real-analytic map from $\mathcal{U}'$ to $\mathbb{R}^g$, which agrees with the function already constructed on $\mathcal{U}$.

\section{An Infinitesimal Criterion}
(Serre gave this argument in his lecture.)

We wish to show that the preimage of $\mathbb{Q}^g$ under $M^{-1} v$ is dense.  Since $\mathcal{U}'$ is open in $\mathbb{R}^{2g}$, we may consider the Jacobian of $M^{-1} v$.  (Here ``Jacobian'' is used in the sense of ``matrix of first partial derivatives,'' not ``Picard scheme.'')

\begin{lemma}
If there is a point $A_0 \in \mathcal{U}'$ at which the Jacobian of $M^{-1} v$ is surjective, then the preimage of $\mathbb{Q}^g$ is dense in $\mathcal{U}'$.
\end{lemma}

\begin{proof}
Let $S$ denote the closure of the preimage under $M^{-1} v$ of $\mathbb{Q}^g$.  We must show $S = \mathcal{U}'$.

First of all, if $A \in \mathcal{U}'$ is any point at which the Jacobian of $M^{-1} v$ is surjective, then by the implicit function theorem $A \in S$.

But $M^{-1} v$ is a real-analytic function on the connected set $\mathcal{U}'$.  The rank of the Jacobian is lower semicontinuous for the analytic Zariski topology, which is to say that the Jacobian has rank $g$ away from a proper analytic subset of $\mathcal{U}$.  Hence, $S$ contains the complement of a proper analytic subset of $\mathcal{U}'$.  But $S$ is closed for the classical topology, hence $S = \mathcal{U}'$.
\end{proof}

We remark in passing that the lemma holds true with $\mathbb{Q}^g$ replaced with any dense subset $X \subseteq \mathbb{R}^g$.  For example, one could take for $X$ the set of all rational numbers with denominator in any infinite set $S$ of positive integers.  The proof is the same.

Now we have reduced the problem to finding a single $A$ where the Jacobian has rank $g$.

\section{A Calculation on $\mathbb{P}^1$}
Choose real numbers
$$-1 < b_1 < b_2 < \cdots < b_g < 1$$
and set
$$f(x) = (x^2 - 1) \prod_i (x - b_i)^2.$$
This $f$ comes from $a \in \mathcal{U}'$ with
$$a_{2i-1} = a_{2i} = b_i$$
Let $\mathcal{V} \subseteq \mathcal{U}'$ denote the $g$-dimensional set of $f$ which arise in this way.

Geometrically, our hyperelliptic curve has degenerated to a $\mathbb{P}^1$ with $g$ nodes.  The global differentials have become differentials with poles on $\mathbb{P}^1$; their integrals turn out to be trigonometric functions.

The space of differentials is now the set of all
$$\frac{p(x) dx} { \left [ \prod_i (x-b_i) \right ] \sqrt{x^2 - 1}},$$
as $p$ ranges over polynomials of degree at most $g-1$.  Recall that the combination $M^{-1} v$ is independent of our choice of basis $\mathcal{B}$.  By partial fractions, we can choose for $\mathcal{B}$ the differentials
$$\omega_i = \frac{dx} {(x-b_i) \sqrt{x^2-1}}.$$

First we compute the ``period matrix.''  The differential $\omega_i$ is meromorphic away from $-1$ and $1$, with a pole only at $b_i$.  So, integrating along $\gamma_j$ for $j \neq i$ gives zero; while
$$\int_{\gamma_i} \omega_i = \frac{2 \pi} {\sqrt{1-b_i^2}}$$
by the residue theorem.  Hence, $M$ is the diagonal matrix with entries
$$\frac{2 \pi} {\sqrt{1-b_i^2}}$$
along the diagonal.

Next we need to compute the integrals
$$v_i = 2 \int_1^\infty \frac{dx} {(x-b_i) \sqrt{x^2-1}}.$$
The substitution $t = x - \sqrt{x^2 - 1}$ transforms the integral to
$$v_i = \int_0^1 \frac{4 dt} {(t^2 - 2b_i t + 1)},$$
which evaluates to
$$v_i = \frac{2 \arccos (-b_i)} {\sqrt{1 - b_i^2}}.$$

Hence, $M^{-1} v$ is the vector with $i$-th entry equal to
$$\frac{\arccos (-b_i)} {\pi}.$$

It follows that the Jacobian of the restriction of $M^{-1} v$ to $\mathcal{V}$ has full rank $g$ at every point in $\mathcal{V}$.  The result follows.

\section*{Acknowledgements}

Akshay Venkatesh suggested the use of Morera's theorem in the proof of Lemma \ref{anal}.  Thanks also to J-P.\ Serre, and to the anonymous referees, for valuable feedback on a draft of this paper.  The author was supported by the Hertz Foundation and the National Science Foundation.

A version of this paper appeared in print as \cite{Law}.

\end{document}